\documentclass[12pt]{article}
\RequirePackage{geometry,amsmath,amsthm,setspace,cite}
\newtheorem{theorem}{Theorem}[section]

\newtheorem{conjecture}[theorem]{Conjecture}
\newtheorem{lemma}[theorem]{Lemma}
\newtheorem{claim}{Claim}

\title{Some bounds on the uniquely restricted matching number}
\author{M. F\"{u}rst  \and D. Rautenbach}
\date{}

\begin{document}
\onehalfspace
\maketitle
\begin{center}
{\small 
Institute of Optimization and Operations Research, Ulm University, Germany\\
\texttt{maximilian.fuerst,dieter.rautenbach@uni-ulm.de}\\[3mm]
}
\end{center}

\begin{abstract}
A matching in a graph is uniquely restricted if no other matching covers exactly the same set of vertices.
We establish tight lower bounds on the maximum size of a uniquely restricted matching in terms of order, size, and maximum degree.
\end{abstract}
 
{\small 
\begin{tabular}{lp{13cm}}
{\bf Keywords:} matching; uniquely restricted matching; acyclic matching
\end{tabular}
}

\section{Introduction}
We consider only simple, finite, and undirected graphs, and use standard terminology.
A {\it matching} \cite{lovasz2009matching} in a graph $G$ is a set of pairwise non-adjacent edges of $G$.
A matching $M$ in $G$ is {\it uniquely restricted} \cite{golumbic2001uniquely} 
if no other matching in $G$ covers exactly the same set of vertices,
and a matching $M$ in $G$ is {\it acyclic} \cite{goddard2005generalized} 
if the subgraph of $G$ induced by the set of vertices of $G$ covered by $M$ is a forest.
The maximum sizes of a matching, a uniquely restricted matching, and an acyclic matching in $G$ 
are denoted by $\nu(G)$, $\nu_{ur}(G)$ and $\nu_{ac}(G)$, respectively.

While unrestricted matchings are tractable \cite{lovasz2009matching},
uniquely restricted and acyclic matchings are both NP-hard even for restricted instances \cite{golumbic2001uniquely, goddard2005generalized},
which motivates the search for tight lower bounds. 
In the present note, 
we prove lower bounds on the uniquely restricted matching number for graphs of bounded maximum degree,
and present improvements for graphs without short cycles.
Similar work has been done for unrestricted, induced, and acyclic matchings 
in general graphs \cite{joos2014induced, joos2016induced, furst2017lower, henning2016tight, henning2007tight, biedl2004tight}
as well as in graphs of large girth \cite{henning2014induced, costa2013matchings, flaxman2007maximum}.
After collecting some notation and known results,
we discuss our contributions. All proofs are postponed to the next section. 
Throughout this section, let $G$ be a graph with vertex set $V(G)$, edge set $E(G)$, and maximum degree at most $\Delta$.
Let $n(G)$, $m(G)$, and $c(G)$ be the order, size, and number of connected components of $G$, respectively. 
The girth of $G$ is the minimum length of a cycle in $G$.
For an integer $k$, let $[k]$ be the set of positive integers at most $k$.

Golumbic et al.~\cite{golumbic2001uniquely} observed that a matching $M$ in $G$ is uniquely restricted 
if and only if there is no $M$-alternating cycle in $G$,
which implies $\nu_{ur}(G) \geq \nu_{ac}(G)$. 
Baste et al.~\cite{baste2017degenerate} showed that the acyclic chromatic index of $G$, 
defined as the minimum number of acyclic matchings into which the edge set of $G$ can be partitioned,
is at most $\Delta^2$. 
Hence, we obtain 
$$\nu_{ur}(G) \geq \nu_{ac}(G) \geq \frac{m(G)}{\Delta^2},$$ 
which is tight for $K_{\Delta,\Delta}$.

As our first result we prove the following.

\begin{theorem} \label{t1}
If $G$ is a graph of maximum degree at most $\Delta$ without isolated vertices, then
\begin{displaymath}
\nu_{ur}(G) \geq \frac{n(G)}{\Delta} - \frac{m(G)}{\Delta^2}
\end{displaymath}
with equality if and only if each component of $G$ is in $\{ K_{\Delta,r}:r\in [\Delta]\}$.
\end{theorem}
In view of the extremal graphs for Theorem \ref{t1},
taking the number of components into account should allow better bounds.
For the subcubic case, that is, for $\Delta=3$, 
we achieve the following best possible result.

\begin{theorem} \label{t2}
If $G$ is a subcubic graph, then
\begin{eqnarray}
\nu_{ur}(G) \geq \frac{n(G) - c(G)}{2} - \frac{m(G)}{6}.\label{et2}
\end{eqnarray}
\end{theorem}
If $G$ is a tree, then (\ref{et2}) simplifies to $\nu_{ur}(G) \geq \frac{n-1}{3}$, 
which is tight for infinitely many trees \cite{henning2016tight}. 
Moreover, 
if the subcubic graph $G_k$
arises from the disjoint union of $k$ isolated vertices $u(1),\ldots,u(k)$ and $2k+1$ copies $K(1),\ldots,K(2k+1)$ of $K_{2,3}$
by adding, for every $i\in [k]$,
one edge between $u(i)$ and $K(2i-1)$,
one edge between $u(i)$ and $K(2i)$, and
one edge between $u(i)$ and $K(2i+1)$,
then 
$$\nu_{ur}(G_k) =  3k+1 = \frac{11k+4}{2} - \frac{15k+6}{6} = \frac{n(G)-1}{2} - \frac{m(G)}{6},$$
that is, (\ref{et2}) is tight also for denser graphs.

If $G$ has girth at least $5$ and no isolated vertices, 
then the approach used in the proof of Theorem \ref{t1} yields $\nu_{ur}(G) \geq \frac{n(G)}{\Delta +1}$.
For $\Delta \geq 4$, we improve this as follows.

\begin{theorem} \label{t3}
If $G$ be a graph of maximum degree at most $\Delta$ for some $\Delta \geq 4$ and girth at least $5$,
then
\begin{eqnarray}
\nu_{ur}(G) \geq \frac{n(G)-c(G)}{\Delta}. \label{et3}
\end{eqnarray}
\end{theorem}
Again, (\ref{et3}) is tight for infinitely many trees \cite{henning2016tight}
but we believe that it still holds if triangles are allowed and if $\Delta\geq 3$. 
It is easy to see that every matching can be partitioned into at most $\Delta$ acyclic matchings, 
which implies
$\nu_{ur}(G) \geq \nu_{ac}(G) \geq \frac{\nu(G)}{\Delta}.$
Baste et al.~\cite{baste2017uniquely} showed that every matching in a bipartite graph with maximum degree $\Delta \geq 4$
not isomorphic to $K_{\Delta,\Delta}$ 
can be partitioned into $\Delta-1$ uniquely restricted matchings,
that is, $\nu_{ur}(G) \geq \frac{\nu(G)}{\Delta-1}$ in this case.
We conjecture that every matching in a graph of maximum degree at most 
$\Delta$ for some $\Delta \geq 3$ and girth at least $7$ 
can be partitioned into at most $\Delta-1$ acyclic matchings. 

A simple counting argument implies that $\nu_{ac}(G) \leq \frac{\Delta n(G)-2}{4\Delta-4}$ for every $\Delta$-regular graph $G$,
that is, regardless of any girth condition, the acyclic matching number can be a factor of about $2$ away from the matching number.
We believe that, at least asymptotically for large girth, 
the uniquely restricted matching number behaves more like the matching number 
than like the acyclic matching number.
More precisely, we conjecture the following.

\begin{conjecture} \label{conjecture2}
For every positive $\epsilon$ and every positive integer $\Delta$, 
there is some positive integer $g=g(\epsilon,\Delta)$, 
such that 
$\nu_{ur}(G)\geq (1-\epsilon)\nu(G)$
for every graph $G$ of maximum degree at most $\Delta$ and girth at least $g$.
\end{conjecture}

\section{Proofs}

In this section, we present the proofs of our theorems 
as well as some auxiliary results.

\begin{proof}[Proof of Theorem \ref{t1}]
Suppose, for a contradiction, that $G$ is a counterexample of minimum order,
that is, 
either $\nu_{ur}(G)<\frac{n(G)}{\Delta} - \frac{m(G)}{\Delta^2}$
or $\nu_{ur}(G)=\frac{n(G)}{\Delta} - \frac{m(G)}{\Delta^2}$
but not all components of $G$ are in $\{ K_{\Delta,r}:r\in [\Delta]\}$.
Let $u$ be a vertex of minimum degree $\delta$.
Let $G^\prime = G - N_G[u]$.
Let $I$ be the set of isolated vertices of $G^\prime$.
Let $E_u$ be the set of edge of $G-u$ that are incident with a vertex in $N_G(u)$.
By the choice of $u$, 
every vertex in $I$ has degree $\delta$ and is completely joined to $N_G(u)$.
Since every vertex in $N_G(u)$ has degree at most $\Delta-1$ in $G'$, we obtain
\begin{eqnarray}\label{e4}
 \delta |I| \leq |E_u| \leq \delta(\Delta -1).
\end{eqnarray}
By the choice of $G$,
and since $M'\cup \{ uv\}$ is a uniquely restricted matching in $G$
for every uniquely restricted matching $M'$ in $G'$ and every neighbor $v$ of $u$,
we obtain
\begin{eqnarray}
  \nu_{ur}(G) &\geq& \nu_{ur}(G^\prime - I) +1\nonumber \\ 
  &\geq& \frac{n(G^\prime - I)}{\Delta} - \frac{m(G^\prime - I)}{\Delta^2} +1 \label{e5}\\
  &=& \frac{1}{\Delta}\Big(n(G) - |I| - \delta - 1\Big) - \frac{1}{\Delta^2}\Big(m(G) - |E_u| - \delta\Big) + 1 \nonumber \\
  &\stackrel{(\ref{e4})}{\geq} & \frac{1}{\Delta}\left(n(G) - \frac{|E_u|}{\delta} - \delta - 1\right) - \frac{1}{\Delta^2}\Big(m(G) - |E_u| - \delta\Big) + 1 \label{e6} \\
  &= &\frac{n(G)}{\Delta} - \frac{m(G)}{\Delta^2} + \frac{\Delta - \delta}{\Delta^2} \left( (\Delta - 1) - \frac{|E_u|}{\delta}\right)\nonumber \\
  &\stackrel{(\ref{e4})}{\geq} & \frac{n(G)}{\Delta} - \frac{m(G)}{\Delta^2}. \label{e7}
\end{eqnarray}
By the choice of $G$, equality holds in (\ref{e5}), (\ref{e6}), and (\ref{e7}).
Equality in (\ref{e5}) implies that every component of $G'-I$ is in $\{ K_{\Delta,r}:r\in [\Delta]\}$.
Equality in (\ref{e7}) implies that $N_G(u)$ is independent, and that every neighbor of $u$ has degree $\Delta$ in $G$.
Finally, equality in (\ref{e6}) implies that $N_G[u]\cup I$ is the vertex set of some component of $G$ that is in $\{ K_{\Delta,r}:r\in [\Delta]\}$.
Altogether, we obtain the contradiction that all components of $G$ are in $\{ K_{\Delta,r}:r\in [\Delta]\}$,
which completes the proof.
\end{proof}
The main result in \cite{furst2017lower} implies 
\begin{eqnarray}
\nu_{ur}(G) \geq \nu_{ac}(G)\geq \frac{n(G)-2}{4}\label{e7}
\end{eqnarray}
for every connected subcubic graph $G$.
We shall use this within the following proof.

\begin{proof}[Proof of Theorem \ref{t2}]
Let $G$ be a counterexample of minimum order $n$. 
Clearly, $G$ is connected and $n\geq 4$. 
Let $m=m(G)$ and $c=c(G)$.
We establish a series of claims.

\begin{claim} \label{c1}
The minimum degree of $G$ is at least $2$.
\end{claim}
\begin{proof}
Suppose, for a contradiction, that $u$ is a vertex of degree $1$.
If $v$ is the unique neighbor of $u$,
then the graph $G'=G-\{u,v \}$ has order $n-2$, size $m-d_G(v)$, and at most $d_G(v)-1$ components.
Since $\{ uv\}\cup M'$ is a uniquely restricted matching in $G$ for every uniquely restricted matching $M'$ in $G'$,
the minimality of $G$ implies
\begin{eqnarray*}
\nu_{ur}(G) &\geq &\nu_{ur}(G')+1
\geq \frac{(n -2) - (d_G(v)-1)}{2} - \frac{m - d_G(v)}{6} + 1
\geq \frac{n - 1}{2} - \frac{m}{6},
\end{eqnarray*}
where we used $d_G(v)\leq 3$ for the last inequality.
This contradiction completes the proof of the claim.
\end{proof}
\begin{claim} \label{c2}
There are no two adjacent vertices of degree $2$.
\end{claim}
\begin{proof}
Suppose, for a contradiction,
that $u$ and $v$ are two adjacent vertices of degree $2$.

First, we assume that 
the edge $uv$ 
either is a bridge 
or belongs to a triangle.
The graph $G'=G-\{ u,v\}$ has order $n-2$, size $m-3$, and at most $2$ components.
Since adding $uv$ to a uniquely restricted matching in $G'$ yields a uniquely restricted matching in $G$,
the minimality of $G$ implies
\begin{eqnarray*}
\nu_{ur}(G) &\geq &\nu_{ur}(G')+1
\geq \frac{(n -2) - 2}{2} - \frac{m- 3}{6} + 1
=\frac{n- 1}{2} - \frac{m}{6}.
\end{eqnarray*}
Hence, we may assume that $uv$ belongs to some cycle $C$ of length at least $4$.

Let $w$ be the neighbor of $u$ distinct from $v$.
By Claim \ref{c1}, the graph $G'=G-\{ u,v,w\}$ has order $n-3$, size at most $m-4$, and at most $2$ components.
Note that adding $uv$ to a uniquely restricted matching in $G'$ yields a uniquely restricted matching in $G$.
If $G'$ is connected, then 
the minimality of $G$ implies
\begin{eqnarray*}
\nu_{ur}(G) &\geq &\nu_{ur}(G')+1
\geq \frac{(n -3) - 1}{2} - \frac{m - 4}{6} + 1
>\frac{n- 1}{2} - \frac{m}{6}.
\end{eqnarray*}
Hence, we may assume that $G'$ has $2$ components.
This implies that $w$ has degree $3$, 
and that a neighbor $x$ of $w$ lies in a different component of $G'$ than the vertices in $V(C)\setminus \{ u,v,w\}$.
By Claim \ref{c1}, the graph $G''=G-\{ u,v,w,x\}$ has order $n-4$, size at most $m-6$, and at most $3$ components.
Since adding $uv$ and $wx$ to a uniquely restricted matching in $G''$ yields a uniquely restricted matching in $G$,
the minimality of $G$ implies
\begin{eqnarray*}
\nu_{ur}(G) &\geq &\nu_{ur}(G'')+2
\geq \frac{(n -4) - 3}{2} - \frac{m - 6}{6} + 2
=\frac{n - 1}{2} - \frac{m}{6},
\end{eqnarray*}
which completes the proof of the claim.
\end{proof}
\begin{claim} \label{c3}
No vertex of degree $2$ is contained in a triangle.
\end{claim}
\begin{proof}
Suppose, for a contradiction, that $uvwu$ is a triangle and that $u$ has degree $2$.
By Claim \ref{c2}, the degree of $v$ and $w$ is $3$.
The graph $G'=G-\{ u,v\}$ has order $n-2$, size $m-4$, and at most $2$ components.
If $G'$ has $2$ components, then,
since adding $uv$ to a uniquely restricted matching in $G'$ yields a uniquely restricted matching in $G$,
the minimality of $G$ implies
\begin{eqnarray*}
\nu_{ur}(G) &\geq &\nu_{ur}(G')+1
\geq \frac{(n -2) - 2}{2} - \frac{m - 4}{6} + 1
>\frac{n - 1}{2} - \frac{m}{6}.
\end{eqnarray*}
Hence, we may assume that $G'$ is connected.
Since $w$ has degree $1$ in $G'$, the graph $G''=G-\{ u,v,w\}$ has order $n-3$, size $m-5$, and is connected.
Since adding $uv$ to a uniquely restricted matching in $G''$ yields a uniquely restricted matching in $G$,
the minimality of $G$ implies
\begin{eqnarray*}
\nu_{ur}(G) &\geq &\nu_{ur}(G'')+1
\geq \frac{(n -3) - 1}{2} - \frac{m - 5}{6} + 1
>\frac{n - 1}{2} - \frac{m}{6},
\end{eqnarray*}
which completes the proof of the claim.
\end{proof}
\begin{claim} \label{c4}
No vertex of degree $2$ is contained in a cycle of length $4$.
\end{claim}
\begin{proof}
Suppose, for a contradiction, that $C:uvxwu$ is a cycle and that $u$ has degree $2$.
By Claim \ref{c3}, the cycle $C$ is induced.
The graph $G'=G-\{ u,v,w\}$ has order $n-3$, size $m-6$, and at most $3$ components.
If $G'$ has at most $2$ components,
then, since adding $uv$ to a uniquely restricted matching in $G'$ yields a uniquely restricted matching in $G$,
the minimality of $G$ implies
\begin{eqnarray*}
\nu_{ur}(G) &\geq &\nu_{ur}(G')+1
\geq \frac{(n -3) - 2}{2} - \frac{m - 6}{6} + 1
=\frac{n- 1}{2} - \frac{m}{6}.
\end{eqnarray*}
Hence, we may assume that $G'$ has $3$ components.
Now, the graph $G''$ that arises from $G$ by contracting $C$ 
is simple, and has order $n-3$, size $m-4$, and is connected.
By the minimality of $G$,
the graph $G''$ has a uniquely restricted matching of size at least $\frac{(n -3) -1}{2} - \frac{m - 4}{6}$.
This implies that $G$ has a uniquely restricted matching $M''$ of size at least $\frac{(n -3) -1}{2} - \frac{m - 4}{6}$
such that $V(M'')$ contains at most one vertex from $C$.
By symmetry, we may assume that $v\not\in V(M'')$.
Now, adding $uv$ to $M''$ yields a uniquely restricted matching in $G$
of size at least 
$$\frac{(n -3) -1}{2} - \frac{m - 4}{6}+1>\frac{n - 1}{2} - \frac{m}{6},$$
which completes the proof of the claim.
\end{proof}
\begin{claim} \label{c5}
For any vertex $u$ of degree $2$, the graph $G-N_G[u]$ has exactly $4$ components. 
\end{claim}
\begin{proof}
Suppose, for a contradiction, that $u$ is a vertex of degree $2$ such that $G-N_G[u]$ has at most $3$ components.
By Claims \ref{c2}, \ref{c3}, and \ref{c4}, 
the two neighbors $v$ and $w$ of $u$ have degree $3$,
are not adjacent, and have no common neighbor.

First, we assume that the edge $uv$ is a bridge.
By symmetry, we may assume that the neighbors of $v$ in $G-N_G[u]$ belong to the same component of that graph.
This implies that the graph $G'=G-\{ u,v\}$ has order $n-2$, size $m-4$, and $2$ components.
Since adding $uv$ to a uniquely restricted matching in $G'$ yields a uniquely restricted matching in $G$,
the minimality of $G$ implies
\begin{eqnarray*}
\nu_{ur}(G) &\geq &\nu_{ur}(G')+1
\geq \frac{(n-2) - 2}{2} - \frac{m- 4}{6} + 1
>\frac{n- 1}{2} - \frac{m}{6}.
\end{eqnarray*}
Hence, we may assume that the edge $uv$ belong to a cycle $C$,
which, by Claim \ref{c4}, has length at least $5$.

Let $H$ be the component of $G-N_G[u]$ that contains $V(C)\setminus \{ u,v,w\}$.
The graph $G''$ that arises from $G$ by contracting all edges of $C\cup H$
has order $n-n(H)-2$, size $m-m(H)-4$, and is connected.
The minimality of $G$ implies that 
the graph $G$ has a uniquely restricted matching $M''$ of size at least $\frac{n-n(H)-2-1}{2}-\frac{m-m(H)-4}{6}$
that used no edge from $C\cup H$ and is incident with at most one vertex in $\{ v,w\}$,
and 
the graph $H$ has a uniquely restricted matching $M_H$ of size at least $\frac{n(H)-1}{2}-\frac{m(H)}{6}$.
By symmetry, we may assume that $v\not\in V(M'')$.
Now, $\{ uv\}\cup M''\cup M_H$ is a uniquely restricted matching in $G$ of size at least 
$$\frac{n-4}{2}-\frac{m-4}{6}+1>\frac{n-1}{2}-\frac{m}{6},$$
which completes the proof of the claim.
\end{proof}
\begin{claim}\label{c6}
$G$ is cubic.
\end{claim}
\begin{proof}
Suppose, for a contradiction, that $u$ is a vertex of degree $2$.
By Claim \ref{c5}, the graph $G-N_G[u]$ has exactly $4$ components.
Let $v$ and $w$ be the neighbors of $u$, and let $x$ be a neighbor of $v$ distinct from $u$.
The graph $G'$ that arises from $G-\{ u,v\}$ by adding the edge $wx$
has order $n-2$, size $m-3$, and $2$ components.
By the minimality of $G$,
the graph $G'$ has a uniquely restricted matching $M'$ of size at least $\frac{(n-2)-2}{2}-\frac{m-3}{6}$.
If $wx\not\in M'$, then let $M=M'\cup \{ uv\}$,
otherwise, let $M=(M'\setminus \{ wx\})\cup \{ uw,vx\}$.
The matching $M$ is uniquely restricted in $G$ and has size at least $\frac{n-1}{2}-\frac{m}{6}$,
which completes the proof of the claim.
\end{proof}
Since $G$ is cubic, we obtain
$\nu_{ur}(G) < \frac{n-1}{2}-\frac{m}{6}=\frac{n-1}{2}-\frac{n}{4}=\frac{n-2}{4}$.
This final contradiction to (\ref{e7})
completes the proof.
\end{proof}
We proceed to improvements for graphs without short cycles.
We deduce Theorem \ref{t3} from the following result.
For a graph $G$, let $n_{\leq 1}(G)$ denote the number of vertices of degree at most $1$ in $G$.

\begin{lemma} \label{t4}
If $G$ is a graph with maximum degree at most $\Delta$ for some $\Delta\geq 4$ and girth at least $5$ 
such that no component of $G$ is $\Delta$-regular, then
\begin{displaymath}
\nu_{ur}(G) \geq \frac{n(G) - n_{\leq 1}(G)}{\Delta}.
\end{displaymath}
\end{lemma}
\begin{proof}
Suppose, for a contradiction, that $G$ is a counterexample of minimum order. 
Let $n=n(G)$ and $n_{\leq 1}=n_{\leq 1}(G)$.
Clearly, $G$ is connected and $n\geq 3$.

\setcounter{claim}{0}

\begin{claim} \label{c7}
The minimum degree of $G$ is at least $2$, that is, $n_{\leq 1}=0$.
\end{claim}
\begin{proof}
Suppose, for a contradiction, that $u$ is a vertex of degree $1$.
Let $v$ be the neighbor of $u$.
The graph $G'=G-\{ u,v\}$ has order $n-2$ and $n_{\leq 1}(G')\leq n_{\leq 1}-1+(\Delta-1)=n_{\leq 1}+\Delta-2$.
Since adding $uv$ to a uniquely restricted matching in $G'$ yields a uniquely restricted matching in $G$,
the minimality of $G$ implies
$$\nu_{ur}(G) \geq \nu_{ur}(G')+1\geq \frac{(n-2)-(n_{\leq 1}+\Delta-2)}{\Delta}+1=\frac{n-n_{\leq 1}}{\Delta},$$
which completes the proof of the claim.
\end{proof}
\begin{claim} \label{c8}
The minimum degree of $G$ is at least $3$.
\end{claim}
\begin{proof}
Suppose, for a contradiction, that $G$ has a vertex of degree $2$.
Let $P: u_1v_1u_2v_2....v_{k-1}u_k$ be a maximal path in $G$ with $d_G(v_i)=2$ for every $i\in [k-1]$.
Since $G$ has a vertex of degree $2$, we obtain $k\geq 2$.
Let $G'=G-V(P)$, and let $V_{\leq 1}'$ be the set of vertices of $G'$ that have degree at most $1$ in $G'$.
By the maximality of $P$ and Claim \ref{c7},
every vertex in $V_{\leq 1}'$ that has a neighbor in $\{ u_1,u_k\}$ has two neighbors in $\{ u_1,\ldots,u_k\}$.
Furthermore, since $G$ has girth at least $5$, at most one vertex in $V_{\leq 1}'$ is adjacent to both $u_1$ and $u_k$.
This implies that all but at most one vertex in $V_{\leq 1}'$ has a neighbor in $\{ u_2,\ldots,u_{k-1}\}$,
which implies
$n_{\leq 1}(G')=|V_{\leq 1}'|\leq 1+(\Delta-2)(k-2)$.
Since adding the edges in $\{ v_iu_{i+1}:i\in [k-1]\}$ to a uniquely restricted matching in $G'$
yields a uniquely restricted matching in $G$, and $\Delta\geq 4$,
the minimality of $G$ implies 
$$\nu_{ur}(G) \geq \nu_{ur}(G')+k-1\geq \frac{(n-(2k-1))-(1+(\Delta-2)(k-2))}{\Delta}+k-1=\frac{n+\Delta-4}{\Delta}\geq \frac{n}{\Delta},$$
which completes the proof of the claim.
\end{proof}
We are now in a position to derive a final contradiction.
Let $u$ be a vertex of minimum degree $\delta$.
By Claim \ref{c8} and since $G$ is not $\Delta$-regular, we obtain $3\leq \delta\leq \Delta-1$.
Let $G'=G-N_G[u]$.
By Claim \ref{c8} and the girth condition, we obtain $n_{\leq 1}(G')=0$.
Since adding an edge incident with $u$ to a uniquely restricted matching in $G'$ yields a uniquely restricted matching in $G$,
the minimality of $G$ implies 
$$\nu_{ur}(G) \geq \nu_{ur}(G')+1\geq \frac{n-(\delta+1)}{\Delta}+1\geq \frac{n}{\Delta},$$
which completes the proof. 
\end{proof}
It is now straightforward to derive Theorem \ref{t3}.

\begin{proof}[Proof of Theorem \ref{t3}]
Suppose, for a contradiction, that $G$ is a counterexample of minimum order. 
Clearly, $G$ is connected and has order at least $3$. 

First, we assume that $G$ has minimum degree $1$.
Let $u$ be a vertex of degree $1$, and let $v$ be its neighbor.
The graph $G'=G-\{u,v\}$ has order $n(G)-2$ and at most $\Delta-1$ components.
Since adding $uv$ to a uniquely restricted matching in $G'$ yields a uniquely restricted matching in $G$,
the minimality of $G$ implies 
$$\nu_{ur}(G)\geq \nu_{ur}(G')+1\geq \frac{(n(G)-2)-(\Delta-1)}{\Delta}+1=\frac{n(G)-1}{\Delta}.$$
Hence, we may assume that $G$ has minimum degree at least $2$.
Furthermore, in view of Lemma \ref{t4}, we may assume that $G$ is $\Delta$-regular.
If $u$ is an endvertex of a spanning tree of $G$,
then $G-u$ is connected and not $\Delta$-regular,
which implies $\nu_{ur}(G) \geq \nu_{ur}(G') \geq \frac{n(G)-1}{\Delta}$.
This final contradiction completes the proof. 
\end{proof}

\end{document}